\documentclass[11pt]{amsart}
\usepackage[usenames,dvipsnames]{pstricks}
\usepackage[mathscr]{eucal}
\usepackage{amsfonts,amsmath,amssymb,amsthm,amscd
,amsxtra}
\usepackage{enumerate,verbatim}
\usepackage[all,2cell,ps]{xy}
\usepackage[notcite, notref]{}
\usepackage[pagebackref]{hyperref}
\usepackage{calc,color}
\usepackage{wasysym}
\usepackage{mathptmx}

\theoremstyle{plain} 

\newtheorem{thm}{Theorem}[section]
\newtheorem{dfn}[thm]{Definition}
\newtheorem{fact}[thm]{Fact}

\newtheorem{rmk}[thm]{Remark}

\newtheorem{cor}[thm]{Corollary}
\newtheorem{lem}[thm]{Lemma}
\newtheorem{con}[thm]{Conventions}

\numberwithin{equation}{section}

\newcommand{\fm}{\mathfrak{m}}

\newcommand{\fp}{\mathfrak{p}}
\newcommand{\fq}{\mathfrak{q}}
\newcommand{\fa}{\mathfrak{a}}
\newcommand{\fb}{\mathfrak{b}}
\newcommand{\fc}{\mathfrak{c}}

\newcommand{\Pc}{\mathcal{P}_C}

\def\ac{\operatorname{\mathcal{A}_{\it C}}}
\def\bc{\operatorname{\mathcal{B}_{\it C}}}
\def\gd{\operatorname{\mathsf{G-dim}}}
\def\Gid{\operatorname{\mathsf{Gid}}}

\def\gkd{\operatorname{\mathsf{G}_{\it C}\mathsf{-dim}}}

\def\gc{\operatorname{\mathsf{G}_{\it C}}}

\def\gcpd{\operatorname{\mathsf{G}_{\it C_p}\mathsf{-dim_{R_p}}}}

\def\pd{\operatorname{\mathsf{pd}}}

\def\gr{\operatorname{\mathsf{grade}}}

\def\Tr{\mathsf{Tr}}
\def\trk{\mathsf{Tr}_{C}}

\def\trc{\mathcal{T}^C}

\def\Min{\mathsf{Min}}

\def\depth{\operatorname{\mathsf{depth}}}
\def\Ext{\operatorname{\mathsf{Ext}}}
\def\Hom{\operatorname{\mathsf{Hom}}}
\def\Tor{\operatorname{\mathsf{Tor}}}

\DeclareMathOperator{\sn}{\widetilde{S}}
\DeclareMathOperator{\ann}{Ann} \DeclareMathOperator{\Ass}{Ass}
 
\DeclareMathOperator{\E}{E} \DeclareMathOperator{\hh}{H}
 
\DeclareMathOperator{\cm}{CM}

\DeclareMathOperator{\coker}{Coker}
\DeclareMathOperator{\Ker}{Ker}

\DeclareMathOperator{\im}{Im}
\DeclareMathOperator{\id}{id}
\DeclareMathOperator{\Supp}{Supp} \DeclareMathOperator{\Spec}{Spec}

\def\urltilda{\kern -.15em\lower .7ex\hbox{\~{}}\kern .04em}
\def\urldot{\kern -.10em.\kern -.10em}\def\urlhttp{http\kern -.10em\lower -.1ex
\hbox{:}\kern -.12em\lower 0ex\hbox{/}\kern -.18em\lower
0ex\hbox{/}}

\begin{document}
\baselineskip=15pt

\title[Linkage of modules with respect to a semidualizing module]
 {Linkage of modules with respect to a semidualizing module}

\bibliographystyle{amsplain}

\author[M. T. Dibaei]{Mohammad T. Dibaei$^1$}\footnotetext[1]{Dibaei was supported in part by a grant from IPM (No. 95130110).} 
\author[A. Sadeghi]{Arash Sadeghi$^2$}\footnotetext[2]{Sadeghi was supported by a grant from IPM.}

\address{$^{1}$ Faculty of Mathematical Sciences and Computer,
Kharazmi University, Tehran, Iran.}

\address{$^{1, 2}$ School of Mathematics, Institute for Research in Fundamental Sciences (IPM), P.O. Box: 19395-5746, Tehran, Iran }
\email{dibaeimt@ipm.ir} \email{sadeghiarash61@gmail.com}


\maketitle

\begin{abstract}
The notion of linkage with respect to a semidualizing module is introduced. It is shown that over a Cohen-Macaulay local ring with canonical module, every Cohen-Macaulay module of finite Gorenstein injective dimension is linked with respect to the canonical module. For a linked module $M$ with respect to a semidualizing module, the connection between the Serre condition $(S_n)$ on $M$ with the vanishing of certain local cohomology modules of its linked module is discussed.
\end{abstract}
\section{Introduction}
The theory of linkage of ideals in commutative algebra has been introduced by Peskine and Szpiro \cite{PS}.
Recall that two ideals $I$ and $J$ in a
Cohen-Macaulay local ring $R$ are said to be linked if there is a
regular sequence $\alpha$ in their intersection such that
$I=(\alpha:J)$ and $J=(\alpha:I)$. One of the main result in the theory of linkage, due to C. Peskine and L. Szpiro, indicates that the Cohen-Macaulayness property is preserved under linkage over Gorenstein local rings.
They also give a counterexample to show that the above result is no longer true if the base ring is Cohen-Macaulay but non-Gorenstein.
Attempts to generalize this theorem has led to several developments in linkage theory, especially by C. Huneke and B. Ulrich (\cite{Hu} and \cite{HuUl}).
In \cite{Sc}, Schenzel used the theory of dualizing complexes to extend the basic properties of linkage to the linkage by Gorenstein ideals.

The classical linkage theory has been extended to modules by Martin \cite{Ma}, Yoshino and Isogawa \cite{YI}, Martsinkovsky and
Strooker \cite{MS}, and by Nagel \cite{N}, in different ways. Based on these generalizations, several works have
been done on studying the linkage theory in the context of modules; see for example \cite{DGHS}, \cite{DS}, \cite{DS1}, \cite{IT}, \cite{Sa} and \cite{CDGST}.
In this paper, we introduce the notion of linkage with respect to a semidualizing module. This is a new notion of linkage for modules and includes the concept of linkage due to Martsinkovsky and Strooker.

To be more precise, let $M$ and $N$ be $R$--modules and let $\alpha$ be an ideal of $R$ which is contained in $\ann_R(M)\cap\ann_R(N)$.
Assume that $K$ is a semidualizing $R/\alpha$--module. We say that $M$ is linked to $N$ with respect to $K$ if
$M\cong\lambda_{R/\alpha}(K,N)$ and $N\cong\lambda_{R/\alpha}(K,M)$, where $$\lambda_{R/\alpha}(K,-):=\Omega_K\Tr_K\Hom_{R/\alpha}(K,-),$$ with $\Omega_K$, $\Tr_K$ are the syzygy and transpose operators, respectively, with respect to $K$. 
This notion enable us to study the theory of linkage for modules in the Bass class with respect to a semidualizing module.
In the first main result of this paper, over a Cohen-Macaulay local ring with canonical module, it is proved that every Cohen-Macaulay module of finite Gorenstein injective dimension is linked with respect to the canonical module (see Theorem \ref{tt}). More precisely,\\

\textbf{Theorem A.}
{\it 
Let $R$ be a Cohen-Macaulay local ring of dimension $d$ with canonical module $\omega_R$. Assume that $\fa$ is a Cohen-Macaulay quasi-Gorenstein ideal of grade $n$ and that
$M$ is a Cohen-Macaulay $R$--module of grade $n$ and of finite Gorenstein injective dimension (equivalently $M\in\mathcal{B}_{\omega_R}$).
If $\fa\subseteq\ann_R(M)$ and $M$ is $\omega_{R/\fa}$-stable, then the following statements hold true.
\begin{itemize}
\item[(i)]{$M$ is linked by ideal $\fa$ with respect to $\omega_{R/\fa}$.}
\item[(ii)]{$\lambda_{R/\fa}(\omega_{R/\fa},M)$ has finite Gorenstein injective dimension.}
\item[(iii)]{$\lambda_{R/\fa}(\omega_{R/\fa},M)$ is Cohen-Macaulay of grade $n$.}\\
\end{itemize}}
Recall that an $R$--module $M$ is called $G$-perfect if $\gr_R(M)=\gd_R(M)$. If $R$ is Cohen-Macaulay then $M$ is $G$-perfect if and only if $M$ is Cohen-Macaulay and $\gd_R(M)<\infty$.
Let us denote the category of $G$-perfect $R$--modules by $\mathcal{X}$, and the category of Cohen-Macaulay $R$-modules of finite Gorenstein injective dimension is by $\mathcal{Y}$. Theorem A, enables us to obtain the following adjoint equivalence (see Theorem \ref{thA} and Theorem \ref{thB}).\\

\textbf{Theorem B.}
{\it 
Let $R$ be a Cohen-Macaulay local ring with canonical module $\omega_R$ and let $\fa$ be a Cohen-Macaulay quasi-Gorenstein ideal,  $\overline{R}=R/\fa$.
There is an adjoint equivalence
\begin{center}

$\left\{\begin{array}{llll}
M\in\mathcal{X} &{\bigg |} \begin{array}{lll} M\mathrm{\ is\ linked\ by}\\  
\mathrm{\ \ \ the\ ideal}\ \fa \\
\end{array} \end{array}
\right\}
\begin{array}{ll}\overset{-\otimes_{\overline{R}}\omega_{\overline{R}}}{\xrightarrow{\hspace*{2cm}}}\\ \underset{\Hom_{\overline{R}}(\omega_{\overline{R}},-)}{\xleftarrow{\hspace*{2cm}}}\\ \end{array}
\left\{\begin{array}{lll}
N\in\mathcal{Y} &{\bigg |} \begin{array}{lll} N\mathrm{\ is\ linked\ by\ the\ ideal }\\ 
\ \ \ \fa\mathrm{\ with\ respect\ to  }\  \omega_{\overline{R}} \end{array}\end{array}
\right\}.$
\end{center}

}
Let $R$ be a Cohen-Macaulay local ring with canonical module $\omega_R$.
For a linked $R$-module $M$, with respect to the canonical module, we study the connection between the Serre condition on $M$ with vanishing of certain local cohomology modules of its linked module. We also establish a duality on local cohomology modules of a linked module which is a generalization of \cite[Theorem 4.1]{Sc} and \cite[Theorem 10]{MS} (see Corollary  \ref{c2} and Corollary \ref{c5}).\\

\textbf{Theorem C.}
{\it 
Let $(R,\fm,k)$ be a Cohen-Macaulay local ring of dimension $d>1$ with canonical module $\omega_R$.
Assume that an $R$--module $M$ is horizontally linked to an $R$--module $N$ with respect to $\omega_R$ and that $M$ has finite Gorenstein injective dimension. Then the following statements hold true.
\begin{itemize}
\item[(i)] { $M$ satisfies $(S_n)$ if and only if $\hh^i_{\fm}(N)=0$ for $d-n<i<d$,}
\item[(ii)]{If $M$ is generalized Cohen-Macaulay then $$\hh^i_{\fm}(\Hom_R(\omega_R,M))\cong\Hom_R(\hh^{d-i}_{\fm}(N),\mathbf{E}_R(k)), \text{ for } 0<i<d.$$ In particular, $N$ is generalized Cohen-Macaulay.}
\end{itemize}}

The organization of the paper is as follows. In Section 2, we collect preliminary notions,
definitions and some known results which will be used in this paper. In Section 3, the precise definition of linkage with respect to a semidualizing is given.
We obtain some necessary conditions for an $R$--module to be linked with respect to a semidualizing (see Theorem \ref{th1}). As a consequence,  we  prove Theorem A and Theorem B in this section. In Section 4, for a linked $R$--module  $M$, with respect to a semidualizing, the relation between the Serre condition $\sn_n$ on $M$ with vanishing of certain relative cohomology modules of its linked module is studied. As a consequence, we prove Theorem C.
\section{Preliminaries}
Throughout the paper, $R$ is a commutative Noetherian semiperfect ring and all
$R$--modules are finitely generated. Note that a commutative ring $R$ is semiperfect if and only if it is
a finite direct product of commutative local rings \cite[Theorem 23.11]{L}. Whenever, $R$ is assumed to be local, its unique maximal ideal
is denoted by $\fm$. The canonical module of $R$ is denoted by $\omega_R$.

Let $M$ be an $R$--module. For a finite projective presentation $P_1\overset{f}{\rightarrow}P_0\rightarrow
M\rightarrow 0$ of $M$, its transpose $\Tr M$ is
defined as $\coker f^*$, where $(-)^* := \Hom_R(-,R)$, which
satisfies in the exact sequence
\begin{equation}\label{1.1}
0\rightarrow M^*\rightarrow P_0^*\overset{f^*}{\rightarrow} P_1^*\rightarrow \Tr
M\rightarrow 0.
\end{equation}
Moreover, $\Tr M$ is unique up to projective equivalence. Thus all
minimal projective presentations of $M$ represent isomorphic
transposes of $M$. The syzygy module $\Omega M$ of $M$ is the kernel of an epimorphism
$P\overset{\alpha}{\rightarrow}M$, where $P$ is a projective
$R$--module which is unique up to projective equivalence. Thus
$\Omega M$ is uniquely determined, up to isomorphism, by a
projective cover of $M$.

Martsinkovsky and Strooker  \cite{MS} generalized the notion of
linkage for modules over non-commutative semiperfect Noetherian
rings (i.e. finitely generated modules over such rings have
projective covers). They introduced the operator
$\lambda:=\Omega\Tr$ and showed that ideals $\fa$ and $\fb$ are
linked by zero ideal if and only if $R/\fa\cong\lambda(R/\fb)$ and
$R/\fb\cong\lambda(R/\fa)$ \cite[Proposition1]{MS}.
\begin{dfn}\label{d1}\cite[Definition 3]{MS}
\emph{Two $R$--modules $M$ and $N$ are said
to be\emph{ horizontally linked} if $M\cong \lambda N$ and
$N\cong\lambda M$. Equivalently, $M$ is
horizontally linked (to $\lambda M$) if and
only if $M\cong\lambda^2 M$.}
\end{dfn}
A \emph{stable} module is a module with no non-zero
projective direct summands. An $R$--module $M$ is called a \emph{syzygy module} if it is
embedded in a projective $R$--module. Let $i$ be a positive integer,
an $R$--module $M$ is said to be an $i$th syzygy if there exists an
exact sequence
$$0\rightarrow M\rightarrow P_{i-1}\rightarrow\cdots\rightarrow P_0$$ with the $P_0,\cdots,P_{i-1}$ are
projective. By convention, every module is a $0$th syzygy.

Here is a characterization of horizontally linked modules.

\begin{thm}\cite[Theorem 2 and Proposition 3]{MS}\label{MS}
An $R$--module $M$ is horizontally
linked if and only if it is stable and $\Ext^1_R(\Tr M,R)=0$,
equivalently $M$ is stable and is a syzygy module.
\end{thm}
Semidualizing modules are initially studied in \cite{F} and \cite{G}.
\begin{dfn}\label{d3}
\emph{An $R$--module $C$ is called a \emph{semidualizing} module, if the
homothety morphism $R\rightarrow\Hom_R(C,C)$ is an isomorphism and
$\Ext^i_R(C,C)=0$ for all $i>0$.}
\end{dfn}
 It is clear that $R$ itself is a semidualizing
$R$--module. Over a Cohen-Macaulay local ring $R$, a canonical
module $\omega_R$ of $R$, if exists, is a semidualizing module with finite
injective dimension.

\begin{con}\label{con}
\emph{ Throughout $C$ denote a semidualizing $R$--module. We set $(-)^\triangledown=\Hom_R(-,C)$ and $(-)^{\curlyvee}=\Hom_{R}(C,-)$.
	The notation $(-)^*$ stands for the $R$--dual functor $\Hom_R(-,R)$. The canonical module of a Cohen-Macaulay local ring, if exists, is denoted as $\omega_R$, then we set $(-)^{\dagger}=\Hom_R(-,\omega_R)$.
}
\end{con}

Let $P_1\overset{f}{\rightarrow}P_0\rightarrow M\rightarrow 0$ be a
projective presentation of an $R$--module $M$. The transpose of $M$ with respect to
$C$, $\trk M$, is defined to be $\coker f^{\triangledown}$, which satisfies in the exact
sequence
\begin{equation}\tag{\ref{d3}.1}
0\rightarrow M^{\triangledown}\rightarrow
P_0^{\triangledown}\overset{f^{\triangledown}}{\rightarrow} P_1^{\triangledown}\rightarrow \trk
M\rightarrow 0.
\end{equation}
By \cite[Proposition 3.1]{F}, there exists the following exact
sequence
\begin{equation}\tag{\ref{d3}.2}
0\rightarrow\Ext^1_R(\trk M,C)\rightarrow M\rightarrow
M^{\triangledown\triangledown}\rightarrow\Ext^2_R(\trk
M,C)\rightarrow0.
\end{equation}

The Gorenstein dimension has been extended to $\gc$--dimension by
Foxby in \cite{F} and Golod in \cite{G}.
\begin{dfn}
\emph{An $R$--module $M$ is said to have \emph{$\gc$--dimension zero} if $M$ is
$C$-reflexive, i.e. the canonical map $M\rightarrow
M^{\triangledown\triangledown}$ is bijective, and
$\Ext^i_R(M,C)=0=\Ext^i_R(M^{\triangledown},C)$ for all $i>0$.}
\end{dfn}
A $\gc$-resolution of an $R$--module $M$ is a right acyclic complex
of $\gc$-dimension zero modules whose $0$th homology is $M$. The
module $M$ is said to have finite $\gc$-dimension, denoted by
$\gkd_R(M)$, if it has a $\gc$-resolution of finite length.

Note that, over a local ring $R$, a semidualizing $R$--module $C$ is
a canonical module if and only if $\gkd_R(M)<\infty$ for all
finitely generated $R$--modules $M$ (see \cite[Proposition
1.3]{Ge}).

In the following, we summarize some basic facts about
$\gc$-dimension (see \cite{AB} and \cite{G} for more details).
\begin{thm}\label{G3}
For an $R$--module $M$, the
following statements hold true.
\begin{itemize}
       \item[(i)]{$\gkd_R(M)=0$ if and only if $\Ext^i_R(M,C)=0=\Ext^i_R(\trk M,C)$ for all $i>0$;}
       \item[(ii)]{$\gkd_R(M)=0$ if and only if $\gkd_R(\trk M)=0$;}
       \item[(iii)]{If $\gkd_R(M)<\infty$ then $\gkd_R(M)=\sup\{i\mid\Ext^i_R(M,C)\neq0, i\geq0\}$;}
        \item[(iv)]{If $R$ is local and $\gkd_R(M)<\infty$, then $\gkd_R(M)=\depth R-\depth_R(M)$.}
\end{itemize}
\end{thm}
The Gorenstein injective dimension, introduced by Enochs and Jenda \cite{EJ}.
\begin{dfn} \label{tanim}\emph{  (\cite{EJ}; see \cite[6.2.2]{Larsbook}) An $R$-module $M$ is said to be\emph{ Gorenstein injective} if there is an exact sequence
$$I_{\bullet}= \cdots \to I_{1} \stackrel{\rm \partial_{1}}{\longrightarrow} I_{0} \stackrel{\rm \partial_{0}}{\longrightarrow}  I_{-1} \to \cdots$$ of injective $R$-modules such that $M\cong \Ker(\partial_{0})$ and $\Hom_{R}(E,I_{\bullet})$ is exact for any injective $R$-module $E$. The Gorenstein injective dimension of $M$, $\Gid(M)$, is defined as the infimum of $n$ for which there exists an exact sequence as $I_\bullet$ with $M\cong\Ker(I_0\rightarrow I_{-1})$ and $I_i=0$ for all $i<-n$.
The Gorenstein injective dimension is a refinement of the classical injective dimension, $\Gid(M)\leq \id(M)$, with equality if $\id(M)<\infty$; see \cite[6.2.6]{Larsbook}. It follows that every module over a Gorenstein ring has finite Gorenstein injective dimension.}
\end{dfn}

\begin{dfn}\label{def1}
\emph{The \emph{Auslander class
with respect to} $C$, $\mathcal{A}_{C}$, consists of all
$R$--modules $M$ satisfying the following conditions.
\begin{itemize}
\item[(i)] The natural map $\mu:M\longrightarrow\Hom_R(C,M\otimes_RC)$ is an isomorphism.
\item[(ii)] $\Tor_i^R(M,C)=0=\Ext^i_R(C,M\otimes_RC)$ for all $i>0$.
\end{itemize}
Dually, the \emph{Bass class} with respect to $C$, $\bc$, consists
of all $R$--modules $M$ satisfying the following conditions.
\begin{itemize}
\item[(i)] The natural evaluation map $\mu:C\otimes_R\Hom_R(C,M)\longrightarrow M$ is an isomorphism.
\item[(ii)] $\Tor_i^R(\Hom_R(C,M),C)=0=\Ext^i_R(C,M)$ for all $i>0$.
\end{itemize}}
\end{dfn}
In the following we collect some basic properties and examples of
modules in the Auslander class, respectively in the Bass class, with respect to $C$ which will be used in the rest of this paper.
\begin{fact}\label{example1} The following statements hold.
\begin{itemize}
  \item[(i)] If any two $R$-modules in a short exact sequence are
     in $\mathcal{A}_C$, respectively $\bc$, then so is the third one \cite[Lemma 1.3]{F}.
     Hence, every module of finite projective dimension
     is in the Auslander class $\mathcal{A}_C$. Also the class $\bc$, contains all modules of finite injective dimension.
  \item[(ii)] Over a Cohen-Macaulay local ring $R$ with canonical module $\omega_R$,
        $M\in\mathcal{A}_{\omega_R}$ if and only if $\gd_R(M)<\infty$
        \cite[Theorem 1]{F1}. Similarly, $M\in\mathcal{B}_{\omega_R}$ if and only if $\Gid_R(M)<\infty$ \cite[Theorem 4.4]{CFH}.
   \item[(iii)] The $\mathcal{P}_C$-projective dimension of $M$, denoted
    $\mathcal{P}_{C}$-$\pd_R(M)$, is less than or equal to $n$ if and
           only if there is an exact sequence
            $$0\rightarrow P_n\otimes_RC\rightarrow\cdots\rightarrow P_0\otimes_RC\rightarrow M\rightarrow0,$$
             such that each $P_i$ is a projective $R$--module \cite[Corollary 2.10]{TW}.
              Note that if $M$ has a finite
              $\mathcal{P}_C$-projective
              dimension, then $M\in\mathcal{B}_C$
               \cite[Corollary 2.9]{TW}.
    \item[(iv)] $M\in\ac$ if and only if $M\otimes_RC\in\bc$. Similarly, $M\in\bc$ if and only if $M^{\curlyvee}\in\ac$ \cite[Theorem 2.8]{TW}.
   
\end{itemize}
\end{fact}
\begin{dfn}\label{def2}\emph{
Let $M$ and $N$ be $R$--modules. Denote by
$\mathcal{\beta}(M, N)$ the set of $R$--homomorphisms of $M$ to $N$
which pass through projective modules. That is, an $R$-homomorphism
$f:M\rightarrow N$ lies in $\mathcal{\beta}(M, N)$ if and only if it
is factored as $M\rightarrow P\rightarrow N$ with $P$ is projective. We
denote the stable homomorphisms from $M$ to $N$ as the quotient module 
$$\underline{\Hom}_R(M,N)=\Hom_R(M,N)/\mathcal{\beta}(M,N).$$
By \cite[Lemma 3.9]{y}, there is a natural isomorphism
\begin{equation}\tag{\ref{def2}.1}
\underline{\Hom}_R(M,N)\cong\Tor_1^R(\Tr M,N).
\end{equation}}
\end{dfn}
The class of $C$--projective modules is defined as 
$$\Pc=\{P\otimes_RC\mid P \text{ is projective}\}.$$
Two $R$--modules $M$ and $N$ are
said to be {\it stably equivalent} with respect to $C$, denoted by $M
\underset{C}{\approx} N$, if $C_1\oplus M\cong C_2\oplus N$ for some $C$--projective
modules $C_1$ and $C_2$. We write $M \approx N$ when $M$
and $N$ are stably equivalent with respect to $R$. An $R$--module $M$ is called {\it $C$--stable} if $M$ does not have a direct summand isomorphic to a $C$--projective module. An $R$--module $M$ is called a $C$-\emph{syzygy module} if it is
embedded in a $C$--projective $R$--module. 

\begin{rmk}\label{remark3} Let $M$ be an $R$--module.
\begin{itemize} 
\item[(i)] Let $P_1\overset{f}{\rightarrow} P_0\rightarrow M\rightarrow0$ be the
minimal projective presentation of $M$. Then $\Tr M\underset{R}{\otimes}C\cong\trk M$ \emph{(see \cite[Remark 2.1(i)]{DS1})}.
\item[(ii)] Note that, by \cite[Proposition 3(a)]{Mar}, $(P_1)^*\rightarrow\Tr M\rightarrow0$ is minimal.
 Therefore, by (i), we get the following exact sequence
 $$0\longrightarrow\Omega_C\trk M\longrightarrow (P_1)^*\otimes_RC\longrightarrow\trk M\longrightarrow0,$$
 where $\Omega_C\trk M:=\im f^{\triangledown}$.
\item[(iii)] It follows, by (\ref{def2}.1), that if $\underline{\Hom}_R(M,C)=0$ then $\Omega_{C}\trk M\cong\lambda M\otimes_R C$.
\end{itemize}
\end{rmk}
\begin{dfn}\cite{M1}\label{S}\emph{
An $R$--module $M$ is said to satisfy the property $\sn_k$
if $\depth_{R_\fp} (M_\fp) \geq  \min\{ k, \depth R_\fp\}$  for all
$\fp\in\Spec R$.}
\end{dfn}
Note that, for a horizontally linked module $M$ over a
Cohen-Macaulay local ring $R$, the properties $\sn_k$ and
$(S_k)$ are identical.
\section{Horizontal linkage with respect to a semidualizing}
In this section $C$ stands for a semidualizing $R$--module and $M$ is an $R$--module.  Set $(M)^\curlyvee:=\Hom_R(C,M)$ as in Convention \ref{con}. In order to develop the notion of linkage with respect to $C$, we give the following definition.
\begin{dfn} \emph{The {\it linkage of $M$ with respect to $C$,} is defined as the module $\lambda_R(C,M):=\Omega_C\trk(M^\curlyvee)$.
The module $M$ is said to be horizontally linked to an $R$--module $N$ with respect to $C$ if $\lambda_R(C, M)\cong N$ and
$\lambda_R(C,N)\cong M$. Equivalently, $M$ is
horizontally linked (to $\lambda_R(C,M)$) with respect to $C$ if and
only if $M\cong\lambda_R^2(C,M)(=\lambda_R(C,\lambda_R(C,M)))$. In this situation $M$ is called a {\it horizontally linked module with respect to $C$}.}
\end{dfn}
Assume that $P_1\overset{f}{\rightarrow} P_0\rightarrow M^{\curlyvee}\rightarrow0$ is the minimal projective presentation of $M^\curlyvee$.
By Remark \ref{remark3}, $\lambda_R(C,M)=\im(f^\triangledown)$ and we obtain the exact sequence
\begin{equation}\label{e}
0\longrightarrow\lambda_R(C,M)\longrightarrow (P_1)^*\otimes_RC\longrightarrow\trk(M^\curlyvee)\longrightarrow0.
\end{equation}
Therefore $\lambda_R(C,M)$ is unique, up to isomorphism.
Having defined the horizontal linkage with respect to a semidualizing module $C$, the general linkage for modules is defined as follows.
\begin{dfn}
\emph{Let $\fa$ be an ideal of $R$ and let $K$ be a semidualizing $R/\fa$-module.
An $R$--module $M$ is said to be \emph{linked} to an $R$--module $N$ by the ideal $\fa$, with respect to $K$, if $\fa\subseteq\ann_R(M)\cap\ann_R(N)$ and $M$ and
$N$ are horizontally
linked with respect to $K$ as $R/{\fa}$--modules. In this situation we denote $M\underset{\fa}{\overset{K}{\thicksim}}N$.} 
\end{dfn}

\begin{lem}\label{lem3}
Assume that an $R$--module $M$ satisfies the
following conditions. 
\begin{itemize}
\item[(i)]  $M$ is a $C$-stable and $C$-syzygy.
\item[(ii)] $\underline{\Hom}_R(M^\curlyvee,C)=0=\underline{\Hom}_R(\lambda(M^\curlyvee),C)$.
\item[(iii)] $M\cong C\otimes_R M^{\curlyvee}$ and $\lambda(M^{\curlyvee})\cong\Hom_R(C,C\otimes_R\lambda(M^{\curlyvee}))$.
\end{itemize}
Then $M$ is a horizontally linked $R$--module with respect to $C$.
\end{lem}
\begin{proof}
As $M$ is $C$-stable, by (iii), $M^\curlyvee$ is stable. By (i), we have the exact sequence $0\rightarrow M\rightarrow P\otimes_RC$ for some projective $R$--module $P$.
By applying the functor $(-)^\curlyvee$ to the above exact sequence, it is easy to see that $M^\curlyvee$ is a first syzygy.
It follows from Theorem \ref{MS} that $M^\curlyvee$ is horizontally linked. In other words, $M^\curlyvee\cong\lambda^2(M^\curlyvee)$. Therefore, we
obtain the following isomorphisms.
\[\begin{array}{rl}
M\cong C\otimes_RM^\curlyvee&\cong
C\otimes_R\lambda^2(M^\curlyvee)\\
&\cong\Omega_C\trk(\lambda(M^\curlyvee))\\
&\cong\Omega_C\trk\Hom_R(C,C\otimes_R\lambda(M^\curlyvee))\\
&\cong\lambda_R(C,\lambda_R(C,M)),
\end{array}\]
by Remark \ref{remark3}(iii) and our assumptions.
\end{proof}
For an integer $n$, set $X^n(R):=\{\fp\in\Spec(R)\mid \depth R_\fp\leq n\}$.
\begin{lem}\label{lem1}
Let $M$ be an $R$--module. Consider the natural map $\mu: M\rightarrow\Hom_R(C,M\otimes_RC)$.
Then the following statements hold true.
\begin{itemize}
\item[(i)] If $M$ satisfies $\sn_1$ and $\mu_\fp$ is a monomorphism for all $\fp\in X^0(R)$, then $\mu$ is a monomorphism.
\item[(ii)]If $M$ satisfies $\sn_2$, $M\otimes_RC$ satisfies $\sn_1$ and $\mu_\fp$ is an isomorphism for all $\fp\in X^1(R)$, then $\mu$ is an isomorphism.
\end{itemize}
\end{lem}
\begin{proof}
(i)  Set $L=\Ker(\mu)$ and let $\fp\in\Ass_R(L)$. Therefore, $\depth_{R_\fp}(M_\fp)=0$. As $M$ satisfies $\sn_1$, $\fp\in X^0(R)$ and so $L_\fp=0$, which is a contradiction.
Therefore, $\mu$ is a monomorphism.

(ii) By (i), $\mu$ is a monomorphism. Consider the following exact sequence:
$$0\longrightarrow M\overset{\mu}{\longrightarrow}\Hom_R(C,M\otimes_RC)\longrightarrow L'\longrightarrow0,$$
where $L':=\coker(\mu)$. Let $\fp\in\Ass_R(L')$. If $\fp\in\Ass_R(\Hom_R(C,M\otimes_RC))\subseteq\Ass_R(M\otimes_RC)$, then $\depth_{R_\fp}(M\otimes_RC)_\fp=0$. As $M\otimes_RC$ satisfies $\sn_1$, one obtains $\fp\in X^0(R)$ which is a contradiction, because  $\mu_\fp$ is an isomorphism for all $\fp\in X^0(R)$. Now let $\depth_{R_\fp}(\Hom_R(C,M\otimes_RC)_\fp)>0$.
It follows easily from the above exact sequence that $\depth_{R_\fp}(M_\fp)=1$. As $M_\fp$ satisfies $\sn_2$, $\fp\in X^1(R)$ which is a contradiction because  $\mu_\fp$ is an isomorphism for all $\fp\in X^1(R)$. Therefore $L'=0$ and $\mu$ is an isomorphism.
\end{proof}
The proof of the following lemma is dual to the proof of \cite[Lemma 2.11]{DS1}.
\begin{lem}\label{lem2}
	Let $R$ be a local ring, $n\geq 0$ an integer, and $M$ an
	$R$--module. If $M\in\mathcal{B}_C$, then the following statements hold true.
	\begin{itemize}
		\item[(i)]  $\depth_R(M)=\depth_R(M^\curlyvee)$ and $\dim_R(M)=\dim_R(M^\curlyvee)$.
		\item[(ii)] $M$ satisfies $\sn_n$ if and only if $M^\curlyvee$ does.
         \item[(iii)] $M$ is Cohen-Macaulay if and only if $M^\curlyvee$ is Cohen-Macaulay.
     	\end{itemize}
\end{lem}
\begin{lem}\cite[Lemma 2.8]{wsw}\label{lwsw}
	Let $M$ be an $R$--module which is in the Bass class $\bc$. Then $\gkd_R(M)=0$ if and only if $\gd_R(M^{\curlyvee})=0$.
\end{lem}
In the following result, we give  sufficient conditions for an element $M\in\bc$ to be  a horizontally linked module with respect to $C$.
\begin{thm}\label{th1}
Assume that $M\in\bc$ is a $C$-syzygy and that $\id_{R_\fp}(C_\fp)<\infty$ for all $\fp\in X^1(R)$. If $M$ is $C$-stable and
$\underline{\Hom}_R(M^\curlyvee,C)=0=\Ext^1_R(M,C)$ then $M$ is a horizontally
linked module with respect to $C$.
\end{thm}
\begin{proof}
We shall prove that the conditions of Lemma \ref{lem3} are satisfied.
First note that 
\begin{equation}\tag{\ref{th1}.1}
M\cong M^{\curlyvee}\otimes C,
\end{equation}
because $M\in\bc$. As, seen in the proof of Lemma \ref{lem3}, $M^\curlyvee$ is horizontally linked. 
In other words, $M^\curlyvee\cong\lambda^2(M^\curlyvee)$ and so 
we obtain the exact sequence
\begin{equation}\tag{\ref{th1}.2}
0\longrightarrow M^\curlyvee\longrightarrow P\longrightarrow\Tr\lambda(M^\curlyvee)\longrightarrow0,
\end{equation}
where $P$ is a projective module. Applying  $-\otimes_RC$ gives the exact sequence
\begin{equation}\tag{\ref{th1}.3}
0\rightarrow\Tor_1^R(\Tr\lambda(M^\curlyvee),C)\rightarrow
M^\curlyvee\otimes_RC\rightarrow P\otimes_RC\rightarrow\Tr\lambda
(M^\curlyvee)\otimes_RC\rightarrow0.
\end{equation}
Let $\fp\in\Ass_R(\Tor_1^R(\Tr\lambda(M^\curlyvee),C))$. It follows from (\ref{th1}.1) and the
exact sequence (\ref{th1}.3) that
$\depth_{R_\fp}(M_\fp)=\depth_{R_\fp}((M^\curlyvee\otimes_RC)_\fp)=0$.
As $M$ is a $C$-syzygy module, $\fp\in X^0(R)$. Note that, by the Fact \ref{example1}(iv), $M^\curlyvee\in\ac$ and so,  $\gd_{R_\fq}((M^\curlyvee)_\fq)=0$ for all $\fq\in X^0(R)$ by the Fact \ref{example1}(ii) and Theorem \ref{G3}(iv).
As $\lambda(M^\curlyvee)$ is a syzygy, one has
\begin{equation}\tag{\ref{th1}.4}
\Ext^1_R(\Tr\lambda(M^\curlyvee),R)=0.
\end{equation}
It follows from (\ref{th1}.4), Theorem \ref{G3} and the exact sequence (\ref{th1}.2) that $\gd_{R_\fp}((\Tr\lambda(M^\curlyvee))_\fp)=0$. In other words, by Fact \ref{example1}(ii), $(\Tr\lambda(M^\curlyvee))_\fp\in\mathcal{A}_{C_\fp}$. Hence $\Tor_1^R(\Tr\lambda(M^\curlyvee),C)_\fp=0$ which is a contradiction. Therefore, $\underline{\Hom}_R(\lambda(M^\curlyvee),C)\cong\Tor_1^R(\Tr\lambda( M^\curlyvee),C)=0$ by (\ref{def2}.1).

Now we prove that the natural map $\mu:\lambda(M^\curlyvee)\rightarrow\Hom_R(C,C\otimes_R\lambda(M^{\curlyvee}))$ is an isomorphism.
To this end, we concentrate on Lemma \ref{lem1}. As $M^\curlyvee$ is horizontally linked, we obtain the isomorphisms
\[\begin{array}{rl}\tag{\ref{th1}.5}
\Ext^2_R(\Tr\lambda(M^\curlyvee),R)&\cong\Ext^1_R(\lambda^2(M^\curlyvee),R)\\
&\cong\Ext^1_R(M^\curlyvee,R)\\
&\cong\Ext^1_R(M^\curlyvee,C^\curlyvee)\\
&\cong\Ext^1_R(M,C)=0,
\end{array}\]
by \cite[Theorem 4.1 and Corollary 4.2]{TW}. It follows from (\ref{th1}.4) and (\ref{th1}.5) that
$\lambda(M^\curlyvee)$ is second syzygy and so it satisfies $\widetilde{S}_2$ by \cite[Proposition 11]{M1}.
By the exact sequence $0\rightarrow\lambda(M^\curlyvee)\rightarrow P'\rightarrow\Tr(M^\curlyvee)\rightarrow0$ and the fact that $\Tor_1^R(\Tr(M^\curlyvee),C)\cong\underline{\Hom}_R(M^\curlyvee,C)=0$, it follows that $\lambda(M^\curlyvee)\otimes_RC$ satisfies $\widetilde{S}_1$.
As $M$ satisfies $\sn_1$, by Fact \ref{example1}(ii), (iv), Lemma \ref{lem2} and Theorem \ref{G3}(iv), $\gd_{R_\fp}((M^\curlyvee)_\fp)=0$ for all $\fp\in X^1(R)$. Therefore, $\gd_{R_\fp}((\lambda(M^\curlyvee))_\fp)=0$ for all $\fp\in X^1(R)$ by \cite[Theorem 1]{MS} and so $(\lambda(M^\curlyvee))_\fp\in\mathcal{A}_{C_\fp}$ for all $\fp\in X^1(R)$ by Fact \ref{example1}(ii). Hence $\mu$ is an isomorphism by Lemma \ref{lem1}.
Now the assertion is clear by Lemma \ref{lem3}.
\end{proof}

In \cite[Corollary 2]{MS}, Martsinkovsky and Strooker proved that, over a
Gorenstein ring, horizontally linkage preserves the property of a module to be
maximal Cohen-Macaulay while they showed, in \cite[Example]{MS}, that over non-Gorenstein rings, being  maximal Cohen-Macaulay need not be preserved under horizontally linkage.
In the following, it is shown that, over a Cohen-Macaulay local ring with the canonical module,
horizontally linkage with respect to canonical module preserves maximal Cohen-Macaulayness. Note that over a Gorenstein ring, every module has finite Gorenstein injective dimension. Therefore, the following result can be viewed as a generalization of \cite[Corollary 2]{MS}. 
\begin{cor}\label{c1}
Let $R$ be a Cohen-Macaulay local ring of with canonical module $\omega_R$. Assume that $M$ is a maximal Cohen-Macaulay $R$--module of finite Gorenstein injective dimension. If $M$ is $\omega_R$-stable then the following statements hold true.
\begin{itemize}
\item[(i)] $M$ is horizontally linked with respect to $\omega_R$.
\item[(ii)] $\lambda_R(\omega_R,M)$ has finite Gorenstein injective dimension.
\item[(iii)] $\lambda_R(\omega_R,M)$ is maximal Cohen-Macaulay.
\end{itemize}
\end{cor}
\begin{proof}
(i) By Fact \ref{example1} (ii), $M\in\mathcal{B}_{\omega_R}$. As $M$ is maximal Cohen-Macaulay, it is a $\omega_R$-syzygy and also $\Ext^1_R(M,\omega_R)=0$. Therefore, by Theorem \ref{th1}, it is enough to prove that $\underline{\Hom}_R(\Hom_R(\omega_R,M),\omega_R)=0$. Note that $\gd_R(\Hom_R(\omega_R,M))=0$ by Theorem \ref{G3} and Lemma \ref{lwsw}.
Hence $\gd_R(\Tr\Hom_R(\omega_R,M))=0$ and $\Tr\Hom_R(\omega_R,M)\in\mathcal{A}_{\omega_R}$ by Fact \ref{example1}(ii) so that $\Tor_i^R(\Tr\Hom_R(\omega_R,M),\omega_R)=0$ for all $i>0$. Indeed, by (\ref{def2}.1)
, $\underline{\Hom}_R(\Hom_R(\omega_R,M),\omega_R)\cong\Tor_1^R(\Tr\Hom_R(\omega_R,M),\omega_R)=0$.
Therefore, by Theorem \ref{th1}, $M$ is horizontally linked with respect to $\omega_R$.

(ii) As we have seen in part (i), $\Tr(\Hom_R(\omega_R,M))\in\mathcal{A}_{\omega_R}$. Hence $\Tr_{\omega_R}(\Hom_R(\omega_R,M))\in\mathcal{B}_{\omega_R}$ by Fact \ref{example1}(iv) and Remark \ref{remark3}(i). Therefore, $\Gid_R(\lambda_R(\omega_R,M))<\infty$ by Fact \ref{example1}(i) and the exact sequence (\ref{e}). 

(iii) By Lemma \ref{lem2}, $\Hom_R(\omega_R,M)$ is maximal Cohen-Macaulay. Therefore 
$\Tr_{\omega_R}(\Hom_R(\omega_R,M))$ is maximal Cohen-Macaulay by Theorem \ref{G3}(ii). It follows from the exact sequence (\ref{e}) that $\lambda_R(\omega_R,M)$ is maximal Cohen-Macaulay.
\end{proof}
To prove Theorem A, we first bring the following lemma and remind a definition.
\begin{lem}\label{lem4}
Let $R$ be a Cohen-Macaulay local ring and let $I$ be an unmixed ideal of $R$. Assume that $K$ is a semidualizing $R/I$--module and that $M$ is an $R$--module which is linked by $I$ with respect to $K$. Then $\gr_R(M)=\gr_R(I)$.
\end{lem}
\begin{proof}
First note that $\gr_R(M)=\inf\{\depth R_\fq\mid\fq\in\Supp_R(M)\}$. Therefore, $\gr_R(M)=\depth R_\fp$ for some $\fp\in\Min_R(M)$ 
 and so $\fp/I\in\Min_{R/I}(M)$.
As $M$ is linked by $I$ with respect to $K$, it is a first $K$-syzygy module and so $\fp/I\in\Ass_{R/I}(R/I)$, because $\Ass_{R/I}(K)=\Ass_{R/I}(R/I)$. As $I$ is unmixed, $\gr(I)=\depth R_\fp$.
\end{proof}
Let $(R,\fm,k)$ be a local ring and let $M$ be an $R$--module. For every integer $n\geq 0$ the $n$th Bass number $\mu^n_R(M)$ is the dimension of the $k$-vector space $\Ext^n_R(k,M)$.
\begin{dfn}\cite{AvF}
	\emph{An ideal $\fa$ of a local $R$ is called \emph{quasi-Gorenstein} if $\gd_R(R/\fa)<\infty$ and for every $i\geq0$ there is an equality of Bass numbers
	$$\mu^{i+\depth R}_R(R)=\mu^{i+\depth R/\fa}_{R/\fa}(R/\fa).$$	}
\end{dfn}

\begin{thm}\cite[Corollary 7.9]{AvF}\label{ex} Let $R$ be a Cohen-Macaulay local ring with canonical module $\omega_R$ and let $\fa$ be a quasi-Gornestein ideal of $R$. For an $R/\fa$--module $M$, $\Gid_{R}(M)<\infty$ if and only if $\Gid_{R/\fa}(M)<\infty$. Also, $\gd_{R}(M)<\infty$ if and only if $\gd_{R/\fa}(M)<\infty$ .
\end{thm}
  We now present Theorem A.
 
\begin{thm}\label{tt}
Let $R$ be a Cohen-Macaulay local ring of dimension $d$ with canonical module $\omega_R$ and let $\fa$ be a Cohen-Macaulay quasi-Gorenstein ideal of grade $n$,  $\overline{R}=R/\fa$. Assume that $M$ is a Cohen-Macaulay $R$--module of grade $n$ and of finite Gorenstein injective dimension
such that $\fa\subseteq\ann_R(M)$. If $M$ is $\omega_{\overline{R}}$-stable then the following statements hold true.
\begin{itemize}
\item[(i)] $M$ is linked by ideal $\fa$ with respect to $\omega_{\overline{R}}$.
\item[(ii)] $\lambda_{\overline{R}}(\omega_{\overline{R}},M)$ has finite Gorenstein injective dimension.
\item[(iii)] $\lambda_{\overline{R}}(\omega_{\overline{R}},M)$ is Cohen-Macaulay of grade $n$.
\end{itemize}
\end{thm}
\begin{proof}
(i)  As $R$ is Cohen-Macaulay,
  $$d-n=d-\gr(\fa)=\dim(R/\fa).$$
On the other hand, as $M$ is Cohen-Macaulay of grade $n$, $$\depth_{\overline{R}}(M)=\depth_R(M)=\dim_R(M)=d-n.$$ Therefore, $M$ is maximal Cohen-Macaulay $\overline{R}$--module.
By Theorem \ref{ex}, $\Gid_{\overline{R}}(M)$ is finite and so $M$ is horizontally linked with respect to $\omega_{\overline{R}}$ as an $\overline{R}$--module by Corollary \ref{c1}.

(ii) By Corollary \ref{c1}, $\Gid_{\overline{R}}(\lambda_{\overline{R}}(\omega_{\overline{R}},M))<\infty$ which is equivalent to say that  $\Gid_{R}(\lambda_{\overline{R}}(\omega_{\overline{R}},M))<\infty$  by Theorem \ref{ex}.

(iii) By Corollary \ref{c1}, $\lambda_{\overline{R}}(\omega_{\overline{R}},M)$ is maximal Cohen-Macaulay $\overline{R}$--module. Hence $$\depth_{R}(\lambda_{\overline{R}}(\omega_{\overline{R}},M))=\depth_{\overline{R}}(\lambda_{\overline{R}}(\omega_{\overline{R}},M))=\dim(R/\fa).$$ Also, by Lemma \ref{lem4}, $\gr_R(\lambda_{\overline{R}}(\omega_{\overline{R}},M))=n$.
Hence, $\dim_{R}(\lambda_{\overline{R}}(\omega_{\overline{R}},M))=d-n=\dim R/\fa$. Therefore, $\lambda_{\overline{R}}(\omega_{\overline{R}},M)$ is Cohen-Macaulay as an $R$--module.
\end{proof}
Let $R$ be a Cohen-Macaulay local ring with canonical module $\omega_R$.
Set $$\mathcal{X}:=\cm(R)\cap\mathcal{A}_{\omega_R}\ \text{and}\ \ 
\mathcal{Y}:=\cm(R)\cap\mathcal{B}_{\omega_R},$$ where $\cm(R)$ is the category of Cohen-Macaulay $R$-module. Now we prove Theorem B.
\begin{thm}\label{thA}
Let $R$ be a Cohen-Macaulay local ring with canonical module $\omega_R$ and let $\fa$ be a Cohen-Macaulay quasi-Gorenstein ideal of grade $n$,  $\overline{R}=R/\fa$.
There is an adjoint equivalence
\begin{center}
	
	$\left\{\begin{array}{llll}
	M\in\mathcal{X} &{\bigg |} \begin{array}{lll}\ \ \  M\mathrm{\ is\ linked }\\  
	\mathrm{\ by\  the\ ideal}\ \fa \\
	\end{array} \end{array}
	\right\}
	\begin{array}{ll}\overset{-\otimes_{\overline{R}}\omega_{\overline{R}}}{\xrightarrow{\hspace*{2cm}}}\\ \underset{\Hom_{\overline{R}}(\omega_{\overline{R}},-)}{\xleftarrow{\hspace*{2cm}}}\\ \end{array}
	\left\{\begin{array}{lll}
	 N\in\mathcal{Y} &{\bigg |} \begin{array}{lll}  N\mathrm{\ is \ linked\ by\ the\ ideal }\\ 
	 \fa \ \mathrm{\ with\ respect\ to  }  \ \omega_{\overline{R}} \end{array}\end{array}
	\right\}.$
\end{center}
\end{thm}
\begin{proof}
Let $M\in\mathcal{X}$, which is linked by the ideal $\fa$. By Theorem \ref{ex}, 
$M\in\mathcal{A}_{\omega_{\overline{R}}}$. Note that $\fa$ is a $G$-perfect ideal and so $\gr_R(M)=\gr_R(\fa)$ by \cite[Lemma 3.16]{Sa}. Therefore 
$$\depth_R(M)=\dim_R(M)=\dim R-\gr_R(M)=\dim R-\gr_R(\fa).$$ 
Hence $M$ is maximal Cohen-Macaulay $\overline{R}$-module.
Set $N=M\otimes_{\overline{R}}\omega_{\overline{R}}$. By \cite[Lemma 2.11]{DS1}, 
$N$ is maximal Cohen-Macaulay $\overline{R}$-module. Therefore $N\in\cm(R)$. Also, by Fact \ref{example1}(iv) and Theorem \ref{ex}, $N\in\mathcal{B}_{\omega_R}$. Hence $N\in\mathcal{Y}$.
As $M\in\mathcal{A}_{\omega_{\overline{R}}}$,
\begin{equation}\tag{\ref{thA}.1}
M\cong\Hom_{\overline{R}}(\omega_{\overline{R}},N).
\end{equation}
Note that $M$ is stable $\overline{R}$-module by Theorem \ref{MS}. It follows from (\ref{thA}.1) that $N$ is $\omega_{\overline{R}}$-stable. Hence, by Theorem \ref{tt}, $N$ 
is linked by the ideal $\fa$ with respect to $\omega_{\overline{R}}$.

Conversely, assume that $N\in\mathcal{Y}$ which is linked by the ideal $\fa$ with respect to $\omega_{\overline{R}}$. As $N$ is Cohen-Macaulay, by Lemma \ref{lem4},
$$\depth_R(N)=\dim_R(N)=\dim R-\gr_R(N)=\dim R-\gr_R(\fa).$$
Therefore $N$ is maximal Cohen-Macaulay $\overline{R}$-module. Set $M=\Hom_{\overline{R}}(\omega_{\overline{R}},N)$. Note that by Theorem \ref{ex}
$N\in\mathcal{B}_{\omega_{\overline{R}}}$. Hence $M\in\mathcal{A}_{\omega_{R}}$ by Fact \ref{example1}(iv), and Theorem \ref{ex}. Also, by Lemma \ref{lem2}, $M$ is maximal Cohen-Macaulay $\overline{R}$-module. Therefore $M\in\mathcal{X}$. Set $X=\Hom_{\overline{R}}(\omega_{\overline{R}},\lambda_{\overline{R}}(\omega_{\overline{R}},N))$. It follows from Theorem \ref{tt}(ii), Fact \ref{example1}(ii), (iv), and Theorem \ref{ex} that $X\in\mathcal{A}_{\omega_{\overline{R}}}$. Also, by Theoerem \ref{tt}(iii) and Lemma \ref{lem2}, $X$ is maximal Cohen-Macaulay $\overline{R}$-module. Therefore, by Theorem \ref{G3}(ii), (iv) and Fact \ref{example1}(ii), $\gd_{\overline{R}}(\lambda_{\overline{R}}X)=0$. In other words,
$\lambda_{\overline{R}}X\in\mathcal{A}_{\omega_{\overline{R}}}$. Hence,
\begin{equation}\tag{\ref{thA}.2}
\lambda_{\overline{R}} X\cong\Hom_{\overline{R}}(\omega_{\overline{R}},\lambda_{\overline{R}} X\otimes_{\overline{R}}\omega_{\overline{R}}).
\end{equation}
As $\lambda_{\overline{R}}X$ is a first syzygy of $\Tr_{\overline{R}}X$, by Fact \ref{example1}(i), $\Tr_{\overline{R}}X\in\mathcal{A}_{\omega_{\overline{R}}}$. Therefore $\underline{\Hom}_{\overline{R}}(X,\omega_{\overline{R}})\cong\Tor_1^{\overline{R}}(\Tr_{\overline{R}}X,\omega_{\overline{R}})=0$. As $N$ is linked by the ideal $\fa$ with respect to $\omega_{\overline{R}}$, it follows from Remark \ref{remark3}(iii) that
\begin{equation}\tag{\ref{thA}.3}
N\cong\Omega_{\omega_{\overline{R}}}\Tr_{\omega_{\overline{R}}}X\cong\lambda_{\overline{R}}X\otimes_{\overline{R}}\omega_{\overline{R}}.
\end{equation}
It follows from (\ref{thA}.2) and (\ref{thA}.3) that $M\cong\lambda_{\overline{R}}X$.
Hence, by \cite[Corollary 1.2.5]{Av} , $M$ is stable $\overline{R}$-module. By \cite[Theorem 1]{MS}, $M$ is linked by the ideal $\fa$.
\end{proof}
Let $\fa$ be an ideal of $R$ an let $M$ be an $R/\fa$-module. Recall that $M$ is said to be \emph{self-linked} by the ideal $\fa$ if $M\cong\lambda_{R/\fa}M$.
Let $K$ be a semidualizing $R/\fa$-module. An $R/\fa$-module $N$ is called \emph{self-linked} by the ideal $\fa$ with respect to $K$ if $N\cong\lambda_{R/\fa}(K,N)$.
\begin{thm}\label{thB}
Let $R$ be a Cohen-Macaulay local ring with canonical module $\omega_R$ and let $\fa$ be a Cohen-Macaulay quasi-Gorenstein ideal of grade $n$,  $\overline{R}=R/\fa$.
There is an adjoint equivalence
$\left\{\begin{array}{lll}
M\in\mathcal{A}_{\omega_R} {\bigg |} \begin{array}{lll} M\mathrm{\ is\ self\ linked }\\  
\mathrm{\ by\ the\ ideal}\hspace{0.25cm} \fa \\
\end{array} \end{array}
\right\}
\begin{array}{ll}\overset{-\otimes_{\overline{R}}\omega_{\overline{R}}}{\xrightarrow{\hspace*{2cm}}}\\ \underset{\Hom_{\overline{R}}(\omega_{\overline{R}},-)}{\xleftarrow{\hspace*{2cm}}}\\ \end{array}
\left\{\begin{array}{lll}
N\in\mathcal{B}_{\omega_R} {\bigg |} \begin{array}{lll} N\mathrm{\ is\ self\ linked\ by\ the\ ideal }\\ 
\fa
\mathrm{\ with\ respect\ to  }\hspace{0.25cm} \omega_{\overline{R}} \end{array}\end{array}
\right\}.$
\end{thm}
\begin{proof}
Let $M\in\mathcal{A}_{\omega_R}$ and let $M\cong\lambda_{\overline{R}}M$.
It follows from the Theorem \ref{ex} that $M\in\mathcal{A}_{\omega_{\overline{R}}}$.
Set $N=M\otimes_{\overline{R}}\omega_{\overline{R}}$.
Therefore,
\begin{equation}\tag{\ref{thB}.1}
M\cong\Hom_{\overline{R}}(\omega_{\overline{R}},N)
\end{equation}
As $M\cong\Omega_{\overline{R}}\Tr_{\overline{R}}M$,
$\Tr_{\overline{R}}M\in\mathcal{A}_{\omega_{\overline{R}}}$.
Hence, $\underline{\Hom}_{\overline{R}}(M,\omega_{\overline{R}})\cong\Tor_1^{\overline{R}}(\Tr_{\overline{R}}M,\omega_{\overline{R}})=0$.
It follows from the (\ref{thB}.1) and Remark \ref{remark3}(iii) that 
\[\begin{array}{rl}
\lambda_{\overline{R}}(\omega_{\overline{R}},N)&=\Omega_{\omega_{\overline{R}}}\Tr_{\omega_{\overline{R}}}(\Hom_{\overline{R}}(\omega_{\overline{R}},N))\\
&\cong\Omega_{\omega_{\overline{R}}}\Tr_{\omega_{\overline{R}}}(M)\\
&\cong\lambda_{\overline{R}}M\otimes_{\overline{R}}\omega_{\overline{R}}\\
&\cong M\otimes_{\overline{R}}\omega_{\overline{R}}=N.
\end{array}\]
In other words, $N$ is self-linked by the ideal $\fa$ with respect to $\omega_{\overline{R}}$. Also, by Fact \ref{example1}(iv), Theorem \ref{ex}, $N\in\mathcal{B}_{\omega_{R}}$.

Conversely, assume that $N\in\mathcal{B}_{\omega_{R}}$ which is self-linked by the ideal $\fa$ with respect to $\omega_{\overline{R}}$. Set $M=\Hom_{\overline{R}}(\omega_{\overline{R}},N)$. It follows from the Fact \ref{example1}(iv), Theorem \ref{ex} that $M\in\mathcal{A}_{\omega_{R}}$. As $N\cong\lambda_{\overline{R}}(\omega_{\overline{R}},N)$, $\Tr_{\omega_{\overline{R}}}(M)\in\mathcal{B}_{\omega_{\overline{R}}}$ by the exact sequence (\ref{e}), Fact \ref{example1}(i) and Theorem \ref{ex}.
It follows from the Remark \ref{remark3}(i) and Fact \ref{example1}(iv) that
$\Tr_{\overline{R}}(M)\in\mathcal{A}_{\omega_{\overline{R}}}$. Therefore $\underline{\Hom}_{\overline{R}}(M,\omega_{\overline{R}})\cong
\Tor_1^{\overline{R}}(\Tr_{\overline{R}}(M)\omega_{\overline{R}})=0$. Hence, by Remark \ref{remark3}(iii),
\begin{equation}\tag{\ref{thB}.2}
N\cong\lambda_{\overline{R}}(\omega_{\overline{R}},N)\cong
\lambda_{\overline{R}}(M)\otimes_{\overline{R}}
\omega_{\overline{R}}.
\end{equation}
As $\Tr_{\overline{R}}(M)\in\mathcal{A}_{\omega_{\overline{R}}}$, $\lambda_{\overline{R}}M\in\mathcal{A}_{\omega_{\overline{R}}}$. Hence
\begin{equation}\tag{\ref{thB}.3}
\lambda_{\overline{R}}M\cong\Hom_{\overline{R}}(\omega_{\overline{R}},
\lambda_{\overline{R}}M\otimes_{\overline{R}}\omega_{\overline{R}})
\end{equation}
It follows from (\ref{thB}.2) and (\ref{thB}.3) that $M\cong\lambda_{\overline{R}}M$. 
\end{proof}
\section{Serre condition and vanishing of local cohomology}
In this Section, for a linked module, we study the relation between
the Serre condition $\sn_n$ with the vanishing of certain relative cohomology modules of its linked module.
As a consequence, the result of Schenzel \cite[Theorem 4.1]{Sc} is generalized. 
We start by the following lemma which will be used in the proof of Theorem \ref{t3}.
\begin{lem}\label{l1}
Let $M$ be a $C$-syzygy module. Then $\Ext^1_R(\trk(M^\curlyvee),C)=0$. In particular, if $M$ is horizontally linked with respect to $C$, then $\Ext^1_R(\trk(M^\curlyvee),C)=0$.
\end{lem}
\begin{proof}
Consider the exact sequence $0\rightarrow M\rightarrow P\otimes_RC$, where $P$ is a projective $R$--module.
Applying the functor $(-)^\curlyvee$ to the above exact sequence, we get the following exact sequence $0\rightarrow M^\curlyvee\rightarrow P$.
Therefore, $\Ext^1_R(\Tr M^\curlyvee,R)=0$. By \cite[Theorem 10.62]{R}, there is a third quadrant spectral sequence
$$\E^{p,q}_2=\Ext^p_R(\Tor_q^R(\Tr(M^{\curlyvee}),C),C)\Rightarrow\Ext^{p+q}_R(\Tr(M^\curlyvee),R).$$
Hence we obtain the following exact sequence $$0\rightarrow\Ext^1_R(\Tr(M^\curlyvee)\otimes_RC,C)\rightarrow\Ext^1_R(\Tr(M^\curlyvee),R),$$
by \cite[Theorem 10.33]{R}. Hence, by Remark \ref{remark3}, $$\Ext^1_R(\trk(M^\curlyvee),C)\cong\Ext^1_R(\Tr(M^\curlyvee)\otimes_RC,C)=0.$$ 
\end{proof}
The following is a generalization of \cite[Theorem 1]{MS}.
\begin{thm}\label{t3}
Let $M$ be an $R$--module which is horizontally linked with respect to $C$. Assume that
$M\in\bc$. Then $\gkd_R(M)=0$ if and only if $\gd_R((\lambda_R(C,M))^\curlyvee)=0$.
\end{thm}
\begin{proof}
Set $N=\lambda_R(C,M)$.	
Consider the following exact sequence 
\begin{equation}\tag{\ref{t3}.1}
0\rightarrow M\rightarrow P^*\otimes_RC\rightarrow\trk(N^\curlyvee)\rightarrow0,
\end{equation}
where $P$ is a pojective $R$--module (see (\ref{e})).
As $M\in\bc$, $\trk(N^\curlyvee)\in\bc$ by the exact sequence (\ref{t3}.1) and Fact (\ref{example1})(i). Hence $\Tr(N^\curlyvee)\in\ac$ by Remark \ref{remark3} and Fact \ref{example1}(iv). In particular,
\begin{equation}\tag{\ref{t3}.2}
\Tr(N^\curlyvee)\cong\Hom_R(C,\Tr(N^\curlyvee)\otimes_RC)\cong\Hom_R(C,\trk(N^\curlyvee)).
\end{equation}
It follows from Theorem \ref{G3}(ii), Lemma \ref{lwsw} and (\ref{t3}.2) that
\begin{equation}\tag{\ref{t3}.3}
\gd_R(N^\curlyvee)=0 \text{ if and only if } \gkd_R(\trk(N^\curlyvee))=0
\end{equation}
On the other hand, by the exact sequence (\ref{t3}.1)
\begin{equation}\tag{\ref{t3}.4}
\gkd_R(M)=0 \text{ and } \Ext^1_R(\trk(N^\curlyvee),C)=0 \text{ if and only if } \gkd_R(\trk(N^\curlyvee))=0.
\end{equation}
Now the assertion is clear by (\ref{t3}.3), (\ref{t3}.4) and Lemma \ref{l1}.
\end{proof}
The class $\Pc$ is precovering and then each $R$--module $M$ has an
augmented proper $\Pc$-resolution, that is,  there is an $R$-complex
$$X^{+}=\cdots\rightarrow C\otimes_RP_1\rightarrow C\otimes_RP_0\rightarrow M\rightarrow0$$
such that $\Hom_R(Y,X^{+})$ is exact for all $Y\in\Pc$. The
truncated complex
$$X=\cdots\rightarrow C\otimes_RP_1\rightarrow C\otimes_RP_0\rightarrow0$$
is called a proper $\Pc$-projective resolution of $M$. Proper $\Pc$-projective
resolutions are unique up to homotopy equivalence.
\begin{dfn}\cite{TW}\label{TWD}
Let $M$ and $N$ be
$R$--modules. The $n$th relative cohomology modules is defined as
$\Ext^n_{\Pc}(M,N)= \hh^n\Hom_R(X,N)$, where $X$ is a proper
$\Pc$-projective resolution of $M$.
\end{dfn}
\begin{thm}\label{t2}\cite[Theorem 4.1 and Corollary 4.2]{TW}
Let $M$ and $N$ be
$R$--modules. Then there exists an isomorphism
$$\Ext^i_{\Pc}(M,N)\cong\Ext^i_R(M^\curlyvee, N^\curlyvee)$$ for all $i\geq0$. Moreover, if $M$ and $N$ are in $\mathcal{B}_C$ then
$\Ext^i_{\Pc}(M,N)\cong\Ext^i_R(M,N)$ for all $i\geq0$.
\end{thm}
For a positive integer $n$, a module $M$ is called an $n$th
$C$-syzygy module if there is an exact sequence $0\rightarrow
M\rightarrow C_1\rightarrow C_2\rightarrow\ldots\rightarrow C_n$,
where $C_i\in\Pc$ for all $i$. The following results will be used in the proof of Theorem \ref{th2}.
\begin{lem}\label{le}
	Let $M$ be an $R$--module such that $\gcpd(M_\fp)<\infty$ for all $\fp\in X^{n-2}(R)$. Then the following statements are equivalent.
	\begin{itemize}
		\item[(i)]{$M$ is an $n$th $C$-syzygy module.}
		\item[(ii)]{$\Ext^i_R(\trk M,C)=0$ for $0<i<n$.}
	\end{itemize}
\end{lem}
\begin{proof}
The proof is analogous to \cite[Theorem 43]{M1}.	
\end{proof}
\begin{thm}\cite[Proposition 2.4]{DS1}\label{ttt1}
	Let $C$ be a semidualizing $R$--module and $M$ an $R$--module. For a
	positive integer $n$, consider the following statements.
	\begin{itemize}
		\item[(i)]$\Ext^i_R(\trk M,C)=0$ for all $i$, $1\leq i\leq n$.
		\item[(ii)]$M$ is an $n$th $C$-syszygy module.
		\item[(iii)]$M$ satisfies $\widetilde{S}_n$.
	\end{itemize}
	Then the following implications hold true.
	\begin{itemize}
		\item[(a)] (i)$\Rightarrow$(ii)$\Rightarrow$(iii).
		\item[(b)] If $M$ has finite $\gc$--dimension on $X^{n-1}(R)$, then (iii)$\Rightarrow$(i).
	\end{itemize}
\end{thm}
The following is a generalization of a result of Schenzel \cite[Theorem 4.1]{Sc}.
\begin{thm}\label{th2}
Let $M$ be an $R$--module which is horizontally linked with respect to $C$. Assume that $M\in\bc$. 
For a positive integer $n$, consider the following statements.
\begin{itemize}
\item[(i)]{$\Ext^i_{\Pc}(\lambda_R(C,M),C)=0$ for $0<i<n$.}
\item[(ii)]{$M$ is an $n$th $C$-syzygy module.}
\item[(iii)]{$M$ satisfies $\sn_n$.}
\end{itemize}
Then the following implications hold true.
\begin{itemize}
	\item[(a)] (i)$\Rightarrow$(ii)$\Rightarrow$(iii).
	\item[(b)] If $M$ has finite $\gc$--dimension on $X^{n-2}(R)$, then the statements (i) and (ii) are equivalent.
	\item[(c)] If $M$ has finite $\gc$--dimension on $X^{n-1}(R)$, then all the statements (i)-(iii) are equivalent.
\end{itemize}
\end{thm}
\begin{proof}
Set $N=\lambda_R(C,M)$. Consider the exact sequence
\begin{equation}\tag{\ref{th2}.1}
0\rightarrow M\rightarrow P\otimes_RC\rightarrow\trk(N^\curlyvee)\rightarrow 0,
\end{equation}
where $P$ is a projective $R$-module. By Lemma \ref{l1}, 
\begin{equation}\tag{\ref{th2}.2}
\Ext^1_R(\trk(N^\curlyvee),C)=0.
\end{equation}
Therefore, by \cite[Lemma 2.2]{DS1}, the exact sequence (\ref{th2}.1) induces the exact 
sequence
\begin{equation}\tag{\ref{th2}.3}
0\rightarrow\trk\trk(N^\curlyvee)\rightarrow Q\otimes_RC\rightarrow\trk M\rightarrow 0,
\end{equation}
where $Q$ is a projective $R$--module.
On the other hand, by \cite[Lemma 2.12]{Sa}, there exists the following exact sequence
\begin{equation}\tag{\ref{th2}.4}
0\rightarrow N^\curlyvee\rightarrow\trk\trk(N^\curlyvee)\rightarrow X\rightarrow 0,
\end{equation}
where $\gkd_R(X)=0$. As $M$ is horizontally linked with respect to $C$,
it is a $C$-syzygy module and so $\Ext^1_R(\trk M,C)=0$. Therefore, by the exact sequences (\ref{th2}.3) and (\ref{th2}.4), we obtain the following.
\begin{equation}\tag{\ref{th2}.5}
\Ext^i_R(\trk M,C)=0 \text{ for } 1\leq i\leq n \Longleftrightarrow \Ext^i_R(N^\curlyvee,C)=0 \text{ for } 1\leq i\leq n-1.
\end{equation}
As $M\in\bc$, by the Fact \ref{example1}(i) and the excat sequence (\ref{th2}.1), $\trk(N^\curlyvee)\in\bc$. Hence, by the Fact \ref{example1}(iv) and Remark \ref{remark3}(i),
$\Tr(N^\curlyvee)\in\ac$. It follows from \cite[Theorem 4.1]{Sa} that
\begin{equation}\tag{\ref{th2}.6}
\Ext^i_R(N^\curlyvee,C)=0 \text{ for } 0<i<n \Longleftrightarrow \Ext^i_R(N^\curlyvee,R)=0 \text{ for } 0<i<n.
\end{equation}
Note that, by Theorem \ref{t2}, we have the isomorphism,
\begin{equation}\tag{\ref{th2}.7}
\Ext^i_{\Pc}(N,C)\cong\Ext^i_R(N^\curlyvee, R) \text{ for all } i\geq0.
\end{equation}

(a), (c). Follow from (\ref{th2}.5), (\ref{th2}.6), (\ref{th2}.7) and Theorem \ref{ttt1}.

(b). Follows from (\ref{th2}.5), (\ref{th2}.6), (\ref{th2}.7) and Lemma \ref{le}.

\end{proof}
\begin{cor}\label{c3}
Let $C$ be a semidualizing $R$--module with $\id_{R_\fp}(C_\fp)<\infty$ for all $\fp\in X^{n-1}(R)$.
Assume that $M$ is an $R$--module which is horizontally linked with respect to $C$ and that $M\in\bc$. Then the following are equivalent.
\begin{itemize}
\item[(i)]{$M$ satisfies $\sn_n$.}
\item[(ii)]{$M$ is an $n$th $C$-syzygy module.}
\item[(iii)]{$\Ext^i_R(\lambda_R(C,M),C)=0$ for $0<i<n$.}
\item[(iv)]{$\Ext^i_{\Pc}(\lambda_R(C,M),C)=0$ for $0<i<n$.}
\end{itemize}
\end{cor}
\begin{proof}
(i)$\Leftrightarrow$(iii) Set $N=\lambda_R(C,M)$.
By Lemma \ref{lem2},
\begin{equation}\tag{\ref{c3}.1}
M \text{ satisfies } \sn_n \text{ if and only if } M^{\curlyvee} \text{ satisfies } \sn_n.
\end{equation}
By Lemma \ref{l1}, $\Ext^1_R(\trk(M^{\curlyvee}),C)=0$. It follows from the exact sequence (\ref{e}) that
\begin{equation}\tag{\ref{c3}.2}
\Ext^i_R(N,C)=0 \text{ for } 0<i<n  \Leftrightarrow \Ext^i_R(\trk(M^{\curlyvee}),C)=0 \text{ for } 0<i<n+1.
\end{equation}
Now the assertion follows from (\ref{c3}.1), (\ref{c3}.2) and Theorem \ref{ttt1}.

The equivalence of (i), (ii) and (iv) follows from Theorem \ref{th2}. 
\end{proof}
Now we are ready to present the first part of Theorem C.
\begin{cor}\label{c2}
Let $(R,\fm)$ be a Cohen-Macaulay local ring of dimension $d>0$ with canonical module $\omega_R$. Assume that $M$ is an $R$--module of finite Gorenstein injective dimension which is horizontally linked with respect to $\omega_R$.
The following are equivalent.
\begin{itemize}
\item[(i)]{$M$ satisfies $(S_n)$.}
\item[(ii)]{$\hh^i_{\fm}(\lambda_R(\omega_R,M))=0$ for $d-n<i<d$.}
\end{itemize}
In particular, $M$ is maximal Cohen-Macaulay if and only if $\lambda_R(\omega_R,M)$ is maximal Cohen-Macaulay.
\end{cor}
\begin{proof}
This is an immediate consequence of Corollary \ref{c3}, Fact \ref{example1}(ii) and the Local Duality Theorem.
\end{proof}
One may translate Corollary \ref{c2} to a change of rings result.
\begin{cor}\label{cc3}
Let $(R,\fm)$ be a Cohen-Macaulay local ring with canonical module $\omega_R$ and let $\fa$ be a Cohen-Macaulay quasi-Gorenstein ideal of $R$ of grade $n$, $\overline{R}=R/\fa$.
Assume that $M$ is an $R$--module of finite Gorenstein injective dimension which is linked by the ideal $\fa$ with respect to $\omega_{\overline{R}}$.
The following are equivalent.
	\begin{itemize}
		\item[(i)]{$M$ satisfies $(S_n)$.}
		\item[(ii)]{$\hh^i_{\fm}(\lambda_{\overline{R}}(\omega_{\overline{R}},M))=0$ for $\dim R/\fa-n<i<\dim R/\fa$.}
	\end{itemize}
\end{cor}
\begin{proof}
This is an immediate consequence of Corollary \ref{c2} and Theorem \ref{ex}.	
\end{proof}
Recall that an $R$--module $M$ of dimension $d\geq1$ is called a
\emph{ generalized Cohen-Macaulay} module if
$\ell(\hh^i_\fm(M))<\infty$ for all $i$, $0\leq i\leq d-1$, where
$\ell$ denotes the length.
For an $R$--module $M$ and positive integer $n$, set $\mathcal{T}_n^{C}M:=\trk\Omega^{n-1}M$. 
\begin{thm}\label{t4}
Let $R$ be a Cohen-Macaulay local ring of dimension $d>1$ and let $C$ be a semidualiznig $R$--module with $\id_{R_\fp}(C_\fp)<\infty$ for all $\fp\in\Spec(R)\setminus\ \{\fm\}$. 
Assume that $M$ is a generalized Cohen-Macaulay $R$--module which is horizontally linked with respect to $C$ and that $M\in\bc$. Then
$\Ext^i_R(M^{\curlyvee},C)\cong\hh^i_{\fm}(\lambda_R(C,M)) \text{ for } 0<i<d.$
In particular, $\lambda_R(C,M)$ is generalized Cohen-Macaulay.
\end{thm}
\begin{proof}
Set $X=M^{\curlyvee}$ and $N=\lambda_R(C,M)$.
As $M$ is generalized Cohen-Macaulay, by \cite[Lemma 1.2, Lemma 1.4]{T} and Theorem \ref{G3}(iv), $\gcpd(M_\fp)=0$ for all $\fp\in\Spec R\setminus \{\fm\}$. Therefore
$\gd_{R_\fp}(X_\fp)=0$ for all $\fp\in\Spec R\setminus \{\fm\}$ by Lemma \ref{lwsw}. Hence, $\Ext^i_R(X,C)$ has finite length for all $i>0$.
Consider the following exact sequences:
\begin{equation}\tag{\ref{t4}.1}
0\rightarrow\Ext^i_R(X,C)\rightarrow\trc_i X\rightarrow L_i\rightarrow0,
\end{equation}
\begin{equation}\tag{\ref{t4}.2}
0\rightarrow L_i\rightarrow\overset{n_i}{\oplus} C\rightarrow\trc_{i+1}X\rightarrow0,
\end{equation}
for all $i>0$. By
applying the functor $\Gamma_{\fm}(-)$ on the exact sequences (\ref{t4}.1) and (\ref{t4}.2), we get
\begin{equation}\tag{\ref{t4}.3}
\hh^j_{\fm}(\trc_{i-1}X)\cong\hh^j_{\fm}(L_{i-1}) \text{ for all}\ i\ \ \text{and}\ j, \text{with}\ j\geq1,\ i\geq2,
\end{equation}
\begin{equation}\tag{\ref{t4}.4}
\Ext^i_R(X,C)=\Gamma_{\fm}(\Ext^i_R(X,C))\cong\Gamma_{\fm}
(\trc_iX) \text{ for all}\ i\geq1,
\end{equation}
and also
\begin{equation}\tag{\ref{t4}.5}
\hh^j_{\fm}(\trc_iX)\cong\hh^{j+1}_{\fm}(L_{i-1})
\text{ for all}\ i\ \text{and}\ j,  0\leq j<d-1,\  i\geq2.
\end{equation}
As $M$ is horizontally-linked with respect to $C$, we have the following exact sequence
$$0\rightarrow N\rightarrow\overset{n}{\oplus} C\rightarrow\trc_1X\rightarrow0,$$
for some integer $n>0$. By applying the functor $\Gamma_{\fm}(-)$ to the above exact sequence, we get the following isomorphism
\begin{equation}\tag{\ref{t4}.6}
\hh^j_{\fm}(\trc_1X)\cong\hh^{j+1}_{\fm}(N)
\text{ for all } j,  0\leq j\leq d-2.
\end{equation}
Now by using (\ref{t4}.3), (\ref{t4}.4), (\ref{t4}.5) and (\ref{t4}.6) we obtain the result.
\end{proof}
Now we give a proof for the part (ii) of Theorem C as the following Corollary.
\begin{cor}\label{c5}
Let $(R,\fm,k)$ be a Cohen-Macaulay local ring of dimension $d>1$ with canonical module $\omega_R$.
Assume that $M$ is an $R$--module of finite Gorenstein injective dimension which is horizontally linked with respect to $\omega_R$.
If $M$ is generalized Cohen-Macaulay then the following statements hold true.
\begin{itemize}
\item[(i)]{$\hh^i_{\fm}(\Hom_R(\omega_R,M))\cong\Hom_R(\hh^{d-i}_{\fm}(\lambda_R(\omega_R,M)),\mathbf{E}_R(k)), \text{ for } 0<i<d.$}
\item[(ii)]{$\lambda_R(\omega_R,M)$ is generalized Cohen-Macaulay.}
\item[(iii)]{If $M$ is not maximal Cohen-Macaulay, then $$\depth_R(M)=\sup\{i<d\mid\hh^i_{\fm}(\lambda_R(\omega_R,M))\neq0\}.$$}
\end{itemize}
\end{cor}
\begin{proof}
Part (i) and (ii) follows immediately from Theorem \ref{t4} and the Local Duality Theorem.
Part (iii) follows from part (i) and Lemma \ref{lem2}.
\end{proof}
We end the paper by the following result which is an immediate consequence of Corollary \ref{c5} and Theorem \ref{ex}.

\begin{cor}
	Let $(R,\fm,k)$ be a Cohen-Macaulay local ring with canonical module $\omega_R$, let $\fc$ be a Cohen-Macaulay quasi-Gorenstein ideal of $R$, $\overline{R}=R/\fc$ and $\dim\overline{R}=d>1$.
	Assume that $M$ is an $R$--module of finite Gorenstein injective dimension which is linked by the ideal $\fc$ with respect to $\omega_{\overline{R}}$.
	If $M$ is generalized Cohen-Macaulay then 
$$\hh^i_{\fm}(\Hom_{\overline{R}}(\omega_{\overline{R}},M))\cong
\Hom_R(\hh^{d-i}_{\fm}(\lambda_{\overline{R}}(\omega_{\overline{R}},M)),\mathbf{E}_R(k)),$$  for $0<i<d.$
\end{cor}

\bibliographystyle{amsplain}

\end{document}